\documentclass[11pt]{amsart}
\usepackage{graphicx} % Required for inserting images
\usepackage{amsmath}
\usepackage{amsfonts}
\usepackage{amsthm}
\usepackage{tikz}
\usepackage[bookmarks,breaklinks,colorlinks=true, allcolors=blue]{hyperref}
\newtheorem{theorem}{Theorem}[section]
\newtheorem{corollary}[theorem]{Corollary}
\newtheorem{lemma}[theorem]{Lemma}

\newtheorem{conjecture}[theorem]{Conjecture}
\newtheorem{proposition}[theorem]{Proposition}

\theoremstyle{definition}

\newtheorem{example}[theorem]{Example}
\newtheorem{remark}[theorem]{Remark}

\usetikzlibrary{arrows, shapes, decorations, decorations.markings, backgrounds, patterns, hobby, knots, calc, positioning}

\pgfdeclarelayer{background}
\pgfdeclarelayer{background2}
\pgfdeclarelayer{background2a}
\pgfdeclarelayer{background2b}
\pgfdeclarelayer{background3}
\pgfdeclarelayer{background4}
\pgfdeclarelayer{background5}
\pgfdeclarelayer{background6}
\pgfdeclarelayer{background7}

\pgfsetlayers{background7,background6,background5,background4,background3,background2b,background2a,background2,background,main}

\definecolor{lsupurple}{RGB}{70,29,124}
\definecolor{lsugold}{RGB}{253,208, 35}

\def\Span{\operatorname{span}}
\def\Span{\operatorname{span}}

\title{A characterization of adequate Turaev genus one links}

\author{Khaled Qazaqzeh}
\address{Department of Mathematics, Faculty of Science, Yarmouk University, Irbid, Jordan, 21163}
\email{qazaqzeh@yu.edu.jo}

\author{Nafaa Chbili}
\address{Department of Mathematical Sciences, College of Science, UAE University, 15551 Al Ain, U.A.E.}
\email{nafaachbili@uaeu.ac.ae}

\author{Adam M. Lowrance}
\address{Department of Mathematics and Statistics, Vassar College, Poughkeepsie, NY 12604}
\email{adlowrance@vassar.edu}

\date{}

\AtBeginDocument{%
   \def\MR#1{}
}

\begin{document}

\begin{abstract}
   We prove that a link is adequate and has Turaev genus one if and only if the span of its Jones polynomial is one less than its crossing number. 
\end{abstract}

\maketitle

\section{Introduction}
\label{sec:intro}

Following Jones' celebrated discovery of the Jones polynomial \cite{Jones_1987}, Kauffman \cite{Kauffman_1987}, Murasugi \cite{Murasugi_1987}, and Thistlethwaite \cite{Thistlethwaite_1987} independently proved that a reduced alternating diagram of a link $L$ has the fewest number of crossings among all diagrams of $L$. Let $V_L(t)$ be the Jones polynomial of $L$, and let $c(L)$ be the crossing number of $L$. A key element in each of their proofs is the inequality $\Span V_L(t) \leq c(L)$, where equality holds if and only if $L$ is alternating. 

Each crossing $\tikz[baseline=.6ex, scale = .4]{
\draw (0,0) -- (1,1);
\draw (.3,.7) -- (0,1);
\draw (.7,.3) -- (1,0);
}
$ of a link diagram $D$ has an $A$-resolution$~ \tikz[baseline=.6ex, scale = .4]{
\draw[rounded corners = 1mm] (0,0) -- (.45,.5) -- (0,1);
\draw[rounded corners = 1mm] (1,0) -- (.55,.5) -- (1,1);
}$ and a $B$-resolution $~ \tikz[baseline=.6ex, scale = .4]{
\draw[rounded corners = 1mm] (0,0) -- (.5,.45) -- (1,0);
\draw[rounded corners = 1mm] (0,1) -- (.5,.55) -- (1,1);
}$. A \textit{Kauffman state} of $D$ is the collection of simple closed curves in the plane resulting from a choice of an $A$-resolution or $B$-resolution at each crossing. The all-$A$ state $\sigma_A(D)$ is the state obtained by choosing an $A$-resolution for every crossing of $D$. Similarly, the all-$B$ state $\sigma_B(D)$ is the state where every resolution is a $B$-resolution. A link diagram $D$ is \textit{$A$-adequate} if no two arcs in the $A$-resolution of any crossing are contained in the same component of the all-$A$ state $\sigma_A(D)$, and similarly, $D$ is \textit{$B$-adequate} if no two arcs in the $B$-resolution of any crossing are contained in the same component of the all-$B$ state $\sigma_B(D)$. Lickorish and Thistlethwaite \cite{LT_1988} define a link to be \textit{adequate} if it has a diagram $D$ that is both $A$-adequate and $B$-adequate.

The Turaev surface of a link diagram $D$ is a closed, oriented surface in $S^3$ of genus 
\begin{equation}
\label{eq:tg}
g_T(D) = \frac{1}{2}(c(D)+2-|\sigma_A(D)|-|\sigma_B(D)|),
\end{equation}
where $c(D)$ is the number of crossings in $D$, and $|\sigma_A(D)|$ and $|\sigma_B(D)|$ are the number of components in the all-$A$ and all-$B$ states of $D$. See Section \ref{sec:background} for a detailed construction of the Turaev surface. The \textit{Turaev genus} $g_T(L)$ of a non-split link $L$ is defined as
\[g_T(L)=\min\{g_T(D)~|~D~\text{is a diagram of}~L\}.\]
Turaev \cite{Turaev_1987} used the surface to give an alternate proof of Kauffman, Murasugi, and Thistlethwaite's theorem on the crossing number of alternating knots. Bae and Morton \cite{BaeMorton_2003} and Dasbach et al. \cite{DFKLS_2008} proved that
\begin{equation}
\label{eq:span}
\Span V_L(t) \leq c(L) - g_T(L),
\end{equation}
where equality holds if $L$ is adequate according to the results in \cite{Abe_2009} and \cite{T_1988_2}.

Our main theorem uses the Jones polynomial to give a characterization of non-split adequate links with Turaev genus one.
\begin{theorem}
    \label{thm:main}
 A non-split link $L$ in $S^3$ is adequate and has Turaev genus one if and only if    
\[\Span V_L(t) = c(L) - 1.\]
\end{theorem}
The forward implication of Theorem \ref{thm:main} follows from the above discussion. We give two distinct proofs of the backward implication in Sections \ref{sec:proof1} and \ref{sec:proof2}, the first using the two variable Kauffman polynomial characterization of adequate links and the second using Bae and Morton's method for computing extreme coefficients of the Jones polynomial \cite{BaeMorton_2003} and a diagrammatic characterization of Turaev genus one links. We hope that one of our proofs may inspire an approach to the following conjecture, which proposes an extension of Theorem \ref{thm:main} to
links of arbitrary Turaev genus.
\begin{conjecture}
\label{conj:main}
A non-split link is adequate if and only if $\Span V_L(t) = c(L) - g_T(L)$.
\end{conjecture}
The first and second authors \cite{QC_2024} attempted to resolve Conjecture \ref{conj:main}, but there is a gap in their proof.  An affirmative answer would give a characterization of an adequate link in terms of its Turaev genus and Jones polynomial. Other characterizations of adequate links include the following. Thistlethwaite \cite{T_1988_2} characterized adequate links in terms of the 2-variable Kauffman polynomial. Kalfagianni \cite{Kalfagianni_2018} and Kalfagianni and Lee \cite{KL_2023} gave related characterizations of adequate links in terms of the colored Jones polynomial. We hope to address Conjecture \ref{conj:main} in a future work.

This paper is organized as follows. Section \ref{sec:background} gives background information on the Jones polynomial, adequate links, and the Turaev surface. In Section \ref{sec:proof1}, we prove Theorem \ref{thm:main} using the two-variable Kauffman polynomial. In Section \ref{sec:proof2}, we prove Theorem \ref{thm:main} using a diagrammatic characterization of Turaev genus one links. Finally, in Section \ref{sec:cor}, we state and prove some corollaries to Theorem \ref{thm:main}.

\section{Background}
\label{sec:background}

In this section, we give background information on the Jones polynomial, adequate knots and links, and on the Turaev surface.
\subsection{The Jones polynomial}

The Kauffman bracket $\langle D \rangle$ of the link diagram $D$ is a Laurent polynomial in $\mathbb{Z}[A,A^{-1}]$ defined by the following rules:
\begin{enumerate}
\item $\left\langle~
\tikz[baseline=.6ex, scale = .4]{
\draw (0,0) -- (1,1);
\draw (1,0) -- (.7,.3);
\draw (.3,.7) -- (0,1);
}
~\right\rangle = A \left\langle ~ \tikz[baseline=.6ex, scale = .4]{
\draw[rounded corners = 1.5mm] (0,0) -- (.45,.5) -- (0,1);
\draw[rounded corners = 1.5mm] (1,0) -- (.55,.5) -- (1,1);
}~\right\rangle + A^{-1} \left\langle~ \tikz[baseline=.6ex, scale = .4]{
\draw[rounded corners = 1.5mm] (0,0) -- (.5,.45) -- (1,0);
\draw[rounded corners = 1.5mm] (0,1) -- (.5,.55) -- (1,1);
}~\right\rangle,$
\item $\left\langle~D\sqcup \bigcirc ~\right\rangle = (-A^2-A^{-2})\left\langle D \right\rangle,$
\item $\left\langle ~ \bigcirc ~\right\rangle = 1.$
\end{enumerate}

Let $M=M(D)=c(D) +2 |\sigma_A(D)|-2$ and $m=m(D)=-c(D)-2|\sigma_B(D)|+2$. Kauffman \cite{Kauffman_1987} proved that $\max\deg\langle D \rangle \leq M$, $\min\deg\langle D \rangle \geq m$, and that powers of $A$ in $\langle D \rangle$ with nonzero coefficients are congruent modulo $4$. Thus the  Kauffman bracket $\langle D \rangle$ can be expressed as
\begin{equation}
    \label{eq:Kauffman}
    \langle D \rangle = a_m A^m + a_{m+4} A^{m+4} + \cdots + a_{M-4}A^{M-4} + a_M A^M.
\end{equation}
The \textit{writhe} $w(D)$ of an oriented link diagram is the number of positive crossings $\left(~\tikz[baseline=.6ex, scale = .4]{
\draw[->] (0,0) -- (1,1);
\draw (1,0) -- (.7,.3);
\draw[->] (.3,.7) -- (0,1);
}~\right)$  in $D$ minus the number of negative crossings $\left(~\tikz[baseline=.6ex, scale = .4]{
\draw[->] (.7,.7) -- (1,1);
\draw[->] (1,0) -- (0,1);
\draw (0,0) -- (.3,.3);
}~\right)$ in $D$. The Jones polynomial $V_L(t)$ of the link $L$ with diagram $D$ is defined as
\[V_L(t) = (-A^3)^{-w(D)}\langle D \rangle |_{A=t^{-1/4}}.\]

The \textit{span} of $V_L(t)$ is $\Span V_L(t) = \max\deg_t V_L(t) - \min\deg_t V_L(t)$. The span of $V_L(t)$ satisfies $4\Span V_L(t) = \Span \langle D \rangle$ where $\Span \langle D \rangle=\max\deg_A \langle D \rangle - \min\deg_A \langle D \rangle$. 

\begin{remark}
    \label{rem:extreme}
    Equation \eqref{eq:Kauffman} implies that
\begin{align*}
    \Span \langle D \rangle \leq & \; M(D) - m(D)\\
    = & \; 2c(D) +2|\sigma_A(D)| + 2|\sigma_B(D)| -4\\
    = & \; 4c(D) - 4g_T(D),
\end{align*}
where equality holds if and only if both coefficients $a_m$ and $a_M$ are non-zero. Minimizing the above inequality over all diagrams $D$ of $L$ implies Inequality \eqref{eq:span}.
\end{remark}

\subsection{Adequate knots and links} Since every reduced alternating diagram is adequate, the set of alternating links is a subset of the set of adequate links. Investigations into adequate links often compare and contrast their properties to the analogous properties of alternating links. Lickorish and Thistlethwaite \cite{LT_1988} defined adequate knots and links and proved every nontrivial adequate link has nontrivial Jones polynomial. Thistlethwaite \cite{T_1988_2} used the Kauffman polynomial to prove that adequate diagrams minimize crossing number, and moreover that every minimum crossing diagram of an adequate link is adequate. Abe \cite{Abe_2009} used the Khovanov homology of a link to show that every adequate diagram minimizes the Turaev genus of a link.

Adequate links and the related class of semi-adequate links have been studied extensively in the literature. The Jones and colored Jones polynomial of adequate and semi-adequate links were studied in \cite{LPS_2001, DL_2006,JRS_2010, St_2011, St_2014, KL_2014, AD_2017,Kalfagianni_2018}. Applications to the geometry and topology of surfaces in the link complement were studied in \cite{FKP_2011, Ozawa_2011, FKP_2013, BMPW_2015, LK_2018}. The Khovanov homology of adequate and semi-adequate links were studied in \cite{Kh_2003, AP_2004, PS_2014, Lee_2018, PS_2020}.

\subsection{The Turaev surface}
Let $D$ be a diagram of the non-split link $L$ in $S^3$ where we interpret $D$ as lying on a projection sphere $S^2\subset S^3$. Embed the all-$A$ and all-$B$ states $\sigma_A(D)$ and $\sigma_B(D)$ in a neighborhood of the projection sphere, but on opposite sides. Construct a cobordism between $\sigma_A(D)$ and $\sigma_B(D)$ that consists of saddles in neighborhoods of the crossings of $D$ (as in Figure \ref{fig:saddle}) and bands outside of neighborhoods of those crossings. The \textit{Turaev surface} $\Sigma_D$ of $D$ is obtained by capping off each boundary component of the cobordism with a disk.
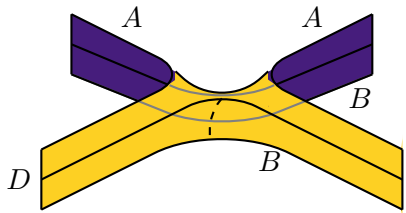
\begin{figure}[h]
\[\begin{tikzpicture}[scale=.8]
\begin{scope}[thick]
\draw [rounded corners = 10mm] (0,0) -- (3,1.5) -- (6,0);
\draw (0,0) -- (0,1);
\draw (6,0) -- (6,1);
\draw [rounded corners = 5mm] (0,1) -- (2.5, 2.25) -- (0.5, 3.25);
\draw [rounded corners = 5mm] (6,1) -- (3.5, 2.25) -- (5.5,3.25);
\draw [rounded corners = 5mm] (0,.5) -- (3,2) -- (6,.5);
\draw [rounded corners = 7mm] (2.23, 2.3) -- (3,1.6) -- (3.77,2.3);
\draw (0.5,3.25) -- (0.5, 2.25);
\draw (5.5,3.25) -- (5.5, 2.25);
\end{scope}

\begin{pgfonlayer}{background2}
\fill [lsugold]  [rounded corners = 10 mm] (0,0) -- (3,1.5) -- (6,0) -- (6,1) -- (3,2) -- (0,1); 
\fill [lsugold] (6,0) -- (6,1) -- (3.9,2.05) -- (4,1);
\fill [lsugold] (0,0) -- (0,1) -- (2.1,2.05) -- (2,1);
\fill [lsugold] (2.23,2.28) --(3.77,2.28) -- (3.77,1.5) -- (2.23,1.5);

\fill [white, rounded corners = 7mm] (2.23,2.3) -- (3,1.6) -- (3.77,2.3);
\fill [lsugold] (2,2) -- (2.3,2.21) -- (2.2, 1.5) -- (2,1.5);
\fill [lsugold] (4,2) -- (3.7, 2.21) -- (3.8,1.5) -- (4,1.5);
\end{pgfonlayer}

\begin{pgfonlayer}{background4}
\fill [lsupurple] (.5,3.25) -- (.5,2.25) -- (3,1.25) -- (2.4,2.2);
\fill [rounded corners = 5mm, lsupurple] (0.5,3.25) -- (2.5,2.25) -- (2,2);
\fill [lsupurple] (5.5,3.25) -- (5.5,2.25) -- (3,1.25) -- (3.6,2.2);
\fill [rounded corners = 5mm, lsupurple] (5.5, 3.25) -- (3.5,2.25) -- (4,2);
\end{pgfonlayer}

\draw [thick] (0.5,2.25) -- (1.6,1.81);
\draw [thick] (5.5,2.25) -- (4.4,1.81);
\draw [thick] (0.5,2.75) -- (2.1,2.08);
\draw [thick] (5.5,2.75) -- (3.9,2.08);

\begin{pgfonlayer}{background}
\draw [black!50!white, rounded corners = 8mm, thick] (0.5, 2.25) -- (3,1.25) -- (5.5,2.25);
\draw [black!50!white, rounded corners = 7mm, thick] (2.13,2.07) -- (3,1.7)  -- (3.87,2.07);
\end{pgfonlayer}
\draw [thick, dashed, rounded corners = 2mm] (3,1.85) -- (2.8,1.6) -- (2.8,1.24);
\draw (0,0.5) node[left]{$D$};
\draw (1.5,3.2) node{$A$};
\draw (4.5,3.2) node{$A$};
\draw (3.8,.8) node{$B$};
\draw (5.3, 1.85) node{$B$};
\end{tikzpicture}\]
    \caption{A saddle in a neighborhood of a crossing of $D$ transitions between the all-$A$ and all-$B$ states of $D$.}
       \label{fig:saddle}
\end{figure}

The genus $g_T(D)$ of the Turaev surface is given in Equation \eqref{eq:tg}. Since $g_T(D)=0$ if and only if $D$ is a connected sum of alternating diagrams, it follows that the Turaev genus of a non-split link $L$ is zero if and only if $L$ is alternating. Hence one can view the Turaev genus of a link as a measure of how non-alternating the link is. 

Dasbach et al. \cite{DFKLS_2008} proved that the Turaev surface is a closed, oriented Heegaard surface on which the link has an alternating projection. However, these properties do not characterize the Turaev surface \cite{Lowrance_2015, ADK_2015,Lowrance_2021}. The Turaev surface of a link diagram has been shown to have connections to the Jones polynomial \cite{BaeMorton_2003,DFKLS_2008,DFKLS_2010}, Khovanov homology \cite{Manturov_2003,Manturov_2006,CKS_2007,DL_2014}, knot Floer homology \cite{Lowrance_2008,DL_2011,JKK_2022}, and the colored Jones polynomial \cite{Kalfagianni_2018,KL_2023}. For further reading, see these two survey articles \cite{CK_2014,KK_2021}.

Computations of Turaev genus are known for several infinite families. As mentioned before, alternating knots have Turaev genus zero. Non-alternating pretzel links and non-alternating Montesinos links have Turaev genus one because their standard diagrams have genus one Turaev surfaces. Abe \cite{Abe_2009} proved that all adequate diagrams are Turaev genus minimizing. Abe and Kishimoto \cite{AK_2010} proved that the Turaev genus of the torus knot $T(3,3n+i)$ is $n$ for $i=1,2$. Jin, Lowrance, Polston, and Zheng \cite{JLPZ_2017} found the Turaev genus of $4$-stranded torus knots and many $5$- and $6$-stranded torus knots. Lowrance \cite{Lowrance_2011} computed the Turaev genus of many closed $3$-braids.

\section{First proof of the main theorem}
\label{sec:proof1}

In this section, we prove Theorem \ref{thm:main} using the two-variable Kauffman polynomial characterization of adequate links. We begin with the definition and results about the Kauffman polynomial. Kauffman \cite{Kauffman_1990} defined a two-variable Laurent polynomial $\Lambda_D(a,z)$ via the following rules:
\begin{enumerate}

\item $\Lambda_{\tikz[baseline=.6ex, scale = .3]{
\draw (0,0) -- (1,1);
\draw (1,0) -- (.7,.3);
\draw (.3,.7) -- (0,1);
}}(a,z) + \Lambda_{\tikz[baseline=.6ex, scale = .3]{
\draw (0,0) -- (.3,.3);
\draw (1,1) -- (.7,.7);
\draw (1,0) -- (0,1);
}}(a,z) = z\left( \Lambda_{\tikz[baseline=.6ex, scale = .3]{
\draw[rounded corners = 1.5mm] (0,0) -- (.45,.5) -- (0,1);
\draw[rounded corners = 1.5mm] (1,0) -- (.55,.5) -- (1,1);
}}(a,z) + \Lambda_{\tikz[baseline=.6ex, scale = .3]{
\draw[rounded corners = 1.5mm] (0,0) -- (.5,.45) -- (1,0);
\draw[rounded corners = 1.5mm] (0,1) -- (.5,.55) -- (1,1);
}} (a,z) \right),$

\item $\Lambda_{\tikz[scale = .3]{
\useasboundingbox (0,0) rectangle (2,1);
\begin{knot}[	
	%draft mode = crossings,
	consider self intersections,
 	clip width = 3,
 	ignore endpoint intersections = true,
	end tolerance = 2pt
	]
	\flipcrossings{1}
	\strand (0,0) to [out = 0 , in = 0, looseness=2]
	(1,1) to [out = 180, in = 180, looseness=2]
	(2,0);
\end{knot}
}}(a,z) = a\; \Lambda_{\tikz[baseline=.6ex, scale = .3]{
\useasboundingbox (0,0) rectangle (1.5,.5);
\draw (0,0) to [out = 30, in = 180] (.75,.5) to [out = 0, in = 150] (1.5,0);
}} (a,z)$, $\Lambda_{\tikz[scale = .3]{
\useasboundingbox (0,0) rectangle (2,1);
\begin{knot}[	
	%draft mode = crossings,
	consider self intersections,
 	clip width = 3,
 	ignore endpoint intersections = true,
	end tolerance = 2pt
	]
	%\flipcrossings{1}
	\strand (0,0) to [out = 0 , in = 0, looseness=2]
	(1,1) to [out = 180, in = 180, looseness=2]
	(2,0);
\end{knot}
}}(a,z) = a^{-1}\; \Lambda_{\tikz[baseline=.6ex, scale = .3]{
\useasboundingbox (0,0) rectangle (1.5,.5);
\draw (0,0) to [out = 30, in = 180] (.75,.5) to [out = 0, in = 150] (1.5,0);
}} (a,z)$,
\item $\Lambda_{\tikz[baseline=.6ex, scale = .3]{
\draw (.4,.4) circle (.4cm);
}} (a,z) = 1$.

\end{enumerate}
He showed that $\Lambda_D(a,z)$ is a regular isotopy invariant, that is, it is invariant under the second and third Reidemeister moves. The \textit{Kauffman polynomial} $F_L(a,z)$ of the link $L$ with a diagram $D$ of writhe $w(D)$ is defined as $F_L(a,z) = a^{-w(D)}\Lambda_D(a,z)$. The Kauffman polynomial $F_L(a,z)$ is an ambient isotopy invariant of oriented links, that is, it is an invariant under all three Reidemeister moves. In this paper, we prefer to use $\Lambda_D(a,z)$ rather than $F_L(a,z)$.

Kauffman \cite{Kauffman_1990} proved that $\Lambda_D(a,z)$ specializes to the Kauffman bracket $\langle D \rangle$ via the formula
\begin{equation}
\label{eq:KauffToKauff}
\langle D \rangle = \Lambda_D(-A^3, A+A^{-1}).
\end{equation}

Thistlethwaite \cite{T_1988_1, T_1988_2} proved the following theorem about the nonzero coefficients of $\Lambda_D(a,z)$.
\begin{theorem}[Thistlethwaite]
\label{thm:KauffmanSupport}
Let $D$ be a link diagram with $c(D)$ crossings, and let $\Lambda_D(a,z) = \sum_{r,s\in\mathbb{Z}} u_{r,s} \;a^r z^s$. 
\begin{enumerate}
\item If $u_{r,s} \neq 0$, then $|r| + s \leq c(D)$.
\item The diagram $D$ is adequate if and only if there are ordered pairs $(r_1,s_1)$ and $(r_2,s_2)$ such that $u_{r_1,s_1}\neq 0$ with $-r_1+s_1=c(D)$ and $u_{r_2,s_2}\neq 0$ with $r_2+s_2=c(D)$. 
\end{enumerate}
\end{theorem}

Define $\Span_a \Lambda_D(a,z) = \max\{r| u_{r,s}\neq 0\} - \min\{r | u_{r,s}\neq 0\}$ where $\Lambda_D(a,z) = \sum u_{r,s} \; a^r z^s$. An \textit{arc presentation} of a link $L$ is an embedding of $L$ into finitely many half-planes whose boundary is the $z$-axis such that each half-plane intersects the link in a single properly embedded arc. The \textit{arc index} $\alpha(L)$ of the link $L$ is the minimum number of half-planes in any arc presentation of $L$. Combining several results about the arc index of a link yields the following result.
\begin{proposition}
\label{prop:aspan}
Let $D$ be a diagram of a non-alternating link. Then $\Span_a \Lambda_D(a,z) \leq c(D)-2$.
\end{proposition}
\begin{proof}
Morton and Beltrami \cite{MB_1998} proved that $\Span_a\Lambda_D(a,z) \leq \alpha(L) - 2$ where $\alpha(L)$ denotes the arc index of the link $L$.  Jin and Park \cite{JP_2010} proved that if $D$ is a minimum crossing diagram of a prime non-alternating link, then $\alpha(L) \leq c(D)=c(L)$. Combining these two inequalities gives the result for prime diagrams. For composite diagrams, the result follows from the additivity of the span of the Kauffman polynomial.
\end{proof}

We are now ready to give our first proof of Theorem \ref{thm:main}. For completeness, we prove both directions of the statement, even though the forward implication was previously known.
\begin{proof}[Proof of Theorem \ref{thm:main}]
Suppose that $L$ is adequate and $g_T(L)=1$. Let $D$ be an adequate diagram of $L$. Thistlethwaite \cite{T_1988_2} proved that $D$ minimizes crossing number, that is $c(D)=c(L)$, and Abe \cite{Abe_2009} proved that $D$ minimizes Turaev genus, that is $g_T(D)=g_T(L)=1$. Lickorish and Thistlethwaite \cite{LT_1988} proved that the extreme coefficients $a_m$ and $a_M$ of $\langle D \rangle$ both have absolute value one. Therefore Remark \ref{rem:extreme} implies $\Span V_L(t) = c(L) - 1$, as desired.

Now suppose that $\Span V_L(t) = c(L) - 1$. Since $\Span V_L(t) < c(L)$, it follows that $L$ is non-alternating. Let $D$ be a minimum crossing diagram of $L$. Since $L$ is non-alternating, Remark \ref{rem:extreme} implies that $g_T(D)=1$ and the extreme coefficients $a_m$ and $a_M$ of $\langle D \rangle$ are non-zero. Since $g_T(D)=1$, it follows that $|\sigma_A(D)| + |\sigma_B(D)| = c(D)$.

Express $\Lambda_D(a,z)$ as the sum $\Lambda_D(a,z) = \sum_{r,s} u_{r,s} \; a^r z^s$. Since $\langle D \rangle = \Lambda_D(-A^3, A+A^{-1})$ by Equation \eqref{eq:KauffToKauff}, it follows that there are ordered pairs $(r_{\max},s_{\max})$ and $(r_{\min},s_{\min})$ with 
$u_{r_{\max},s_{\max}}\neq 0$ and $u_{r_{\min},s_{\min}}\neq 0$ satisfying the inequalities
\begin{align}
	3r_{\max}+s_{\max} & \; \geq M~\text{and} \label{ineq:max1}\\
	3r_{\min} - s_{\min} & \; \leq m \label{ineq:min1}.
\end{align}
Part (1) of Theorem \ref{thm:KauffmanSupport} implies that the following inequalities hold:
\begin{align}
	r_{\max} + s_{\max} & \; \leq c(D)~\text{and} \label{ineq:max2}\\
	-r_{\min} + s_{\min} & \; \leq c(D). \label{ineq:min2}
\end{align}
Inequalities \eqref{ineq:max1} and \eqref{ineq:max2} imply that $2r_{\max} \geq M - c(D)= 2|\sigma_A(D)|-2$, and Inequalities \eqref{ineq:min1} and \eqref{ineq:min2} imply that $2r_{\min} \leq m + c(D) = -2|\sigma_B(D)|+2$. Hence $r_{\max} - r_{\min} \geq |\sigma_A(D)| + |\sigma_B(D)| -2 = c(D) -2$. Since Proposition \ref{prop:aspan} implies that $r_{\max}-r_{\min} \leq c(D)-2$, it follows that $r_{\max}-r_{\min} = c(D)-2$. This forces $r_{\max} = |\sigma_{A}(D)| -1$ and $r_{\min} = 1- |\sigma_{B}(D)|$. Inequalities \eqref{ineq:max1} and \eqref{ineq:max2} imply that $s_{\max} = c(D) - |\sigma_{A}(D)|+1$, and Inequalities \eqref{ineq:min1} and \eqref{ineq:min2} imply that $s_{\min} = c(D) - |\sigma_{B}(D)|+1$.  Since equality holds in Inequalities \eqref{ineq:max2} and \eqref{ineq:min2}, Part (2) of Theorem \ref{thm:KauffmanSupport} implies that $D$ is adequate, as desired. 
\end{proof}

The above proof of Theorem \ref{thm:main} can potentially be generalized to a proof of Conjecture \ref{conj:main}, but such a generalization requires the following strengthening of Proposition \ref{prop:aspan}.
\begin{conjecture}
\label{conj:aspan}
If a link $L$ has crossing number $c(L)$, Turaev genus $g_T(L)$, and diagram $D$, then
\[\Span F_L(a,z) = \Span \Lambda_D(a,z) \leq c(L) - 2g_T(L).\]
\end{conjecture}

Yokota \cite{Yokota_1995} proved that $\Span F_L(a,z) = c(L)$ when $L$ is alternating, implying Conjecture \ref{conj:aspan} for alternating links. Proposition \ref{prop:aspan} implies the conjecture for Turaev genus one links. Using data from \cite{knotinfo}, we have confirmed Conjecture \ref{conj:aspan} for all knots with at most twelve crossings.

\section{Second proof of Theorem \ref{thm:main}}
\label{sec:proof2}

In this section, we give our second proof of Theorem \ref{thm:main}. First, we need a few additional results about extreme Jones coefficients and also about links with Turaev genus one.

Bae and Morton \cite{BaeMorton_2003} described the following procedure that uses the all-$A$ and all-$B$ states to compute the extreme coefficients $a_M$ and $a_m$ of $\langle D \rangle$ respectively. We focus on the procedure to compute $a_M$. Construct the decorated all-$A$ state by replacing each crossing $\tikz[baseline=.6ex, scale = .4]{
\draw (0,0) -- (1,1);
\draw (1,0) -- (.7,.3);
\draw (.3,.7) -- (0,1);
}$ with its $A$-resolution and a line segment connecting the two arcs of the $A$-resolution $
\tikz[baseline=.6ex, scale = .4]{
\draw[red] (.3,.5) -- (.7,.5);
\draw[rounded corners = 1.5mm] (0,0) -- (.45,.5) -- (0,1);
\draw[rounded corners = 1.5mm] (1,0) -- (.55,.5) -- (1,1);
}$. In general, the line segment is called the \textit{trace} of the crossing. If both endpoints of the trace lie on the same component of $\sigma_A(D)$, then we say that trace is an \textit{$A$-chord}. Let $\widetilde{\sigma_A}(D)$ be the all-$A$ state, decorated with its $A$-chords (and where the non-chord traces are deleted/ignored). A subset $C$ of $A$-chords in $\widetilde{\sigma_A}(D)$ is \textit{independent} if the endpoints of each pair of $A$-chords that lie on the same component do not alternate in order around that component. The empty set of chords is declared to be independent.
\begin{theorem}[Bae, Morton]
    \label{thm:extreme}
    The coefficient $a_M$ of the maximum degree term in $\langle D \rangle$ is
    \[a_M = (-1)^{|\sigma_A(D)|-1} \sum (-1)^{|C|}\]
    where the sum is over all independent subsets $C$ of $A$-chords in the decorated all-$A$ state $\widetilde{\sigma_A}(D)$. Similarly, the coefficient $a_m$ of the minimum degree term in $\langle D \rangle$ is given by
    \[a_m = (-1)^{|\sigma_B(D)|-1} \sum (-1)^{|C|}\]
    where the sum is over all independent subsets $C$ of $B$-chords in the decorated all-$B$ state $\widetilde{\sigma_B}(D)$.
\end{theorem}

Manch\'on \cite{Manchon_2004} continued the study of extreme Jones coefficients. Several groups of authors used a similar approach to study the extreme terms in Khovanov homology \cite{GMMS_2018, PS_2018, PS_2020, FS_2021,BGMVRS_2025}. 

\begin{example}
    \label{ex:BM}
    Let $D$ be the diagram in Figure \ref{fig:BM}. Its decorated all-$A$ state $\widetilde{\sigma_A}(D)$ has three $A$-chords, labeled $1$, $2$, and $3$. The subsets of independent $A$-chords are $\emptyset$, $\{1\}$, $\{2\}$, $\{3\}$, and $\{1,2\}$. Theorem \ref{thm:extreme} implies 
    \[a_M = (-1)^{9-1}(1 -3 +1) = -1.\]
    There are no $B$-chords in $\widetilde{\sigma_B}(D)$. Hence $D$ is $B$-adequate, the only independent subset of $B$-chords is empty, and Theorem \ref{thm:extreme} implies 
    \[a_m = (-1)^{6-1}(1) = -1.\]
\end{example}
\begin{figure}[h]
\[\begin{tikzpicture}[scale=.4]

\begin{knot}[ 	
	%draft mode = crossings,
	consider self intersections,
 	clip width = 6,
 	ignore endpoint intersections = true,
	end tolerance = 2pt
	]
	\flipcrossings{2,4,6,8,9,11,13,15}
	\strand[thick, rounded corners = 1mm]  (3,0) -- (0,0) -- (0,1) -- (1,2) -- (2,1) -- (3,2) -- (1,4) -- (3,6) -- (2,7) -- (1,6) --(0,7) -- (0,8) -- (9.5,8) -- (9.5,5.5) -- (8.5,4.5) -- (9.5,3.5) -- (8.5,2.5) -- (5.5,5.5) -- (4.5,4.5) --  (5.5,3.5) -- (3,1) -- (2,2) -- (1,1) -- (0,2) -- (2,4) -- (0,6) -- (1,7) -- (2,6) -- (3,7) -- (5.5,4.5) -- (4.5,3.5) -- (5.5,2.5) -- (8.5,5.5) -- (9.5,4.5) -- (8.5,3.5) -- (9.5,2.5) -- (9.5,0) -- (3,0);
	
\end{knot}

\draw (4.75,-1) node{$D$};

\begin{scope}[xshift = 10.5cm]

	\draw[thick,red] (1.2,4.5) -- (1.8,4.5);
	\draw[thick,red] (1.2,3.5) -- (1.8,3.5);
	\draw[thick,red] (7,3.7) -- (7,4.3);
	\draw[red] (0.6,4.5) node{\color{red}{\small{$1$}}};
	\draw[red] (0.6,3.5) node{\color{red}{\small{$2$}}};
	\draw[red] (7,4.9) node{\color{red}{\small{$3$}}};
	\draw[thick,red,densely dotted] (.2,1.5) -- (.7,1.5);
	\draw[thick,red,densely dotted] (1.3,1.5) -- (1.7,1.5);
	\draw[thick,red,densely dotted] (2.3,1.5) -- (2.8,1.5);
	\draw[thick,red,densely dotted] (.2,6.5) -- (.7,6.5);
	\draw[thick,red,densely dotted] (1.3,6.5) -- (1.7,6.5);
	\draw[thick,red,densely dotted] (2.3,6.5) -- (2.8,6.5);
	\draw[thick,red,densely dotted] (5,4.8) -- (5,5.3);
	\draw[thick,red, densely dotted] (5,4.2) -- (5,3.8);
	\draw[thick,red, densely dotted] (5,3.2) -- (5,2.7);
	\draw[thick,red,densely dotted] (9,4.8) -- (9,5.3);
	\draw[thick,red, densely dotted] (9,4.2) -- (9,3.8);
	\draw[thick,red, densely dotted] (9,3.2) -- (9,2.7);

	\draw[thick,rounded corners = 1mm] (3,0) -- (0,0) -- (0,1) -- (.3,1.5) -- (0,2) -- (1,3) -- (1.3,3.5) -- (1,4) -- (1.3,4.5) -- (1,5) -- (0, 6) -- (.3,6.5) -- (0,7) -- (0,8) -- (9.5,8) -- (9.5,5.5) -- (9,5.2) -- (8.5,5.5) -- (7,4.2) -- (5.5,5.5) -- (5,5.2) -- (4.5,5.5) -- (3,7) -- (2.7,6.5) -- (3,6) -- (2,5) -- (1.7,4.5) -- ( 2,4) -- (1.7,3.5) -- ( 2,3) -- (3,2) -- (2.7,1.5) -- (3,1) -- (4.5,2.5) -- (5,2.8) -- (5.5,2.5) -- (7,3.8) -- ( 8.5,2.5) -- (9,2.8) -- (9.5,2) -- (9.5,0) -- (3,0);
	\draw[thick] (1,1.5) circle (.3cm);
	\draw[thick] (2,1.5) circle (.3cm);
	\draw[thick] (1,6.5) circle (.3cm);
	\draw[thick] (2,6.5) circle (.3cm);
	\draw[thick] (5,3.5) circle (.3cm);
	\draw[thick] (5,4.5) circle (.3cm);
	\draw[thick] (9,3.5) circle (.3cm);
	\draw[thick] (9,4.5) circle (.3cm);
	
	\draw (4.75,-1) node{$\widetilde{\sigma_A}(D)$};

\end{scope}

\begin{scope}[xshift = 21cm]

	\draw[thick,blue,densely dotted] (.5,1.2) -- (.5,1.85);
	\draw[thick,blue,densely dotted] (1.5,1.2) -- (1.5,1.85);
	\draw[thick,blue,densely dotted] (2.5,1.2) -- (2.5,1.85);
	\draw[thick,blue,densely dotted] (.5,6.15) -- (.5,6.8);
	\draw[thick,blue,densely dotted] (1.5,6.15) -- (1.5,6.8);
	\draw[thick,blue,densely dotted] (2.5,6.15) -- (2.5,6.8);
	\draw[thick,blue,densely dotted] (1.5,3.7) -- (1.5,3.15);
	\draw[thick,blue,densely dotted] (1.5,4.3) -- (1.5,4.85);
	\draw[thick,blue,densely dotted] (4.75,4) -- (5.25,4);
	\draw[thick,blue,densely dotted] (4.75,5) -- (5.25,5);
	\draw[thick,blue,densely dotted] (4.75,3) -- (5.25,3);
	\draw[thick,blue,densely dotted] (8.75,4) -- (9.25,4);
	\draw[thick,blue,densely dotted] (8.75,5) -- (9.25,5);
	\draw[thick,blue,densely dotted] (8.75,3) -- (9.25,3);
	\draw[thick,blue,densely dotted] (6.75,4) -- (7.25,4);

	\draw[thick, rounded corners = 1mm] (3,0) -- (0,0) -- (0,1) -- (0.5,1.3) -- (1,1)  -- (1.5,1.3) -- (2,1) -- (2.5,1.3) -- (3,1) -- (4.5,2.5) -- (4.8,3) -- (4.5,3.5) -- (4.8,4) -- (4.5,4.5) -- (4.8,5) -- (4.5,5.5)  -- (3,7) -- (2.5,6.7) -- (2,7) -- (1.5,6.7) -- (1,7) -- (0.5,6.7) -- (0,7) -- (0,8) -- (9.5,8) -- (9.5,5.5) -- (9.2,5) -- (9.5,4.5) -- (9.2,4) -- (9.5,3.5)  -- (9.2,3) -- (9.5,2.5)-- (9.5,0) -- (3,0);
	
	\draw[thick, rounded corners = 1mm] (6.5, 3.5) -- (6.8,4)  -- (6.5,4.5) -- (5.5,5.5) -- (5.2,5) -- (5.5,4.5) -- (5.2,4) -- (5.5,3.5) -- (5.2,3) -- (5.5,2.5) -- (6.5,3.5);
	\draw[thick, rounded corners = 1mm] (7.5,3.5) -- (7.2,4) -- (7.5,4.5) -- (8.5,5.5) -- (8.8,5) -- (8.5,4.5) -- (8.8,4) -- ( 8.5,3.5) -- (8.8,3) -- ( 8.5,2.5) -- (7.5,3.5);
	\draw[thick, rounded corners = 1mm] (0.5,5.5) -- (1,5) -- (1.5,4.8) -- (2,5) -- (3,6) -- (2.5,6.2) -- (2,6) -- (1.5,6.2) -- ( 1,6) -- ( .5,6.2) -- (0,6) -- (0.5,5.5);
	\draw[thick] (1.5,4) circle (.3cm);
	\draw[thick, rounded corners = 1mm] (0.5,2.5) -- (1,3) -- (1.5,3.2) -- (2,3) -- (3,2) -- (2.5,1.8) -- (2,2) -- (1.5,1.8) -- (1,2) -- (.5,1.8) -- (0,2) -- (0.5,2.5);
	\draw (4.75,-1) node{$\widetilde{\sigma_B}(D)$};

\end{scope}

\end{tikzpicture}\]
\caption{A link diagram $D$ and its decorated all-$A$ and all-$B$ states $\widetilde{\sigma_A}(D)$ and $\widetilde{\sigma_B}(D)$. Chords are drawn with solid lines while traces that are not chords are drawn with dotted lines.}
\label{fig:BM}
\end{figure}
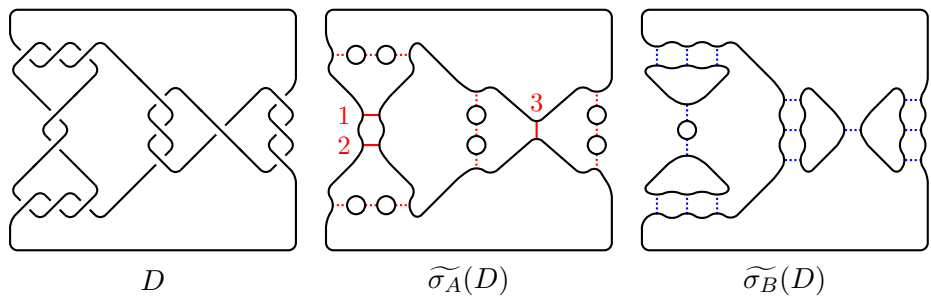

Remark \ref{rem:extreme} implies that if $D$ is a diagram such that $\Span V_L(t) = c(D) - g_T(D)$, then the extreme coefficients $a_m$ and $a_M$ must be non-zero. The following lemma describes a scenario where $a_M$ is zero. By replacing ``$A$" with ``$B$" throughout the lemma, one obtains conditions ensuring $a_m=0$.
\begin{lemma}
\label{lem:extremecoeff}
    Suppose that the decorated all-$A$ state $\widetilde{\sigma_A}(D)$ has at least one $A$-chord. Furthermore, suppose that for each component $\gamma$ of $\widetilde{\sigma_A}(D)$ all $A$-chords meeting $\gamma$ are on the interior of $\gamma$ or all $A$-chords meeting $\gamma$ are on the exterior of $\gamma$. Then the maximum coefficient $a_M$ of $\langle D \rangle$ is zero.
\end{lemma}
\begin{proof}
    Since every $A$-chord meeting a given component is on one side of the curve, it follows that every subset of $A$-chords is independent. Therefore Theorem \ref{thm:extreme} implies 
    \[a_M = (-1)^{|\sigma_A(D)|-1}\sum_{k=0}^n (-1)^k\binom{n}{k}=0.\]
\end{proof}

The classification of link diagrams whose Turaev surface has genus one in terms of their alternating tangle decompositions is an important piece of our second proof of Theorem \ref{thm:main}. The alternating tangle decomposition of a link diagram $D$ is constructed as follows. Consider the link diagram $D$ as a $4$-regular graph whose vertices are the crossings. An edge of $D$ is \textit{alternating} if one of its endpoints is an under-crossing and one of its endpoints is an over-crossing, and an edge of $D$ is \textit{non-alternating} if both of its endpoints are under-crossings or both of its endpoints are over-crossings. Mark each non-alternating edge of $D$ with two points. Inside each face of $D$, draw arcs that connect marked points that are adjacent on the boundary of the face but do not lie on the same non-alternating edge. These arcs form a collection of simple closed curves dividing the diagram $D$ into alternating tangles connected by non-alternating edges. See Figure \ref{fig:altdecomp} for an example.
\begin{figure}[h]
\[\begin{tikzpicture}
\begin{scope}[scale = .8]
\begin{scope}[thick]
	\draw (.7,0) -- (2.2,0);
	\draw (2.5,0) -- (2.8,0);
	\draw (2.1,-.3) -- (3.6,1.2);
	\draw (3.4,1.2) -- (3.4,2.2);
	\draw (3.4, .8) -- (3.4,.6);
	\draw (3.4, 2.6) -- (3.4,2.8);
	\draw (3.6,2.2) -- (2.6,3.2);
	\draw (2.3,3.5) -- (2.1,3.7);
	\draw (2.7,3.4) -- (.7,3.4);
	\draw (.9,3.3) -- (.1,2.5);
	\draw (1.1,3.5) -- (1.3,3.7);
	\draw (-.1,2.3) -- (-.3,2.1);
	\draw (0,2.7) -- (0,.7);
	\draw (0.1,.9) -- (.9,.1);
	\draw (-.1,1.1) -- (-.3,1.3);
	\draw (1.1,-.1) -- (1.3,-.3);
\end{scope}

\node (1) at (.7,.3){};
\node (2) at (.3,.7){};
\node (3) at (0,1.5) {};
\node (4) at (0,1.9) {};
\node (5) at (.3,2.7) {};
\node (6) at (.7,3.1) {};
\node (7) at (1.5,3.4) {};
\node (8) at (1.9, 3.4) {};
\node (9) at (3.4,1.9) {};
\node (10) at (3.4,1.5) {};
\node (11) at (3.1,.7) {};
\node (12) at (2.7,.3) {};

\begin{scope}[line/.style={shorten >=-0.2cm,shorten <=-0.2cm},thick,green!70!black]
\fill (1) circle (.1cm);
\fill (2) circle (.1cm);
\fill (3) circle (.1cm);
\fill (4) circle (.1cm);
\fill (5) circle (.1cm);
\fill (6) circle (.1cm);
\fill (7) circle (.1cm);
\fill (8) circle (.1cm);
\fill (9) circle (.1cm);
\fill (10) circle (.1cm);
\fill (11) circle (.1cm);
\fill (12) circle (.1cm);
\path [bend left, line]   (1) edge (12);
\path [bend left, line]   (3) edge (2);
\path [bend left, line]   (5) edge (4);
\path [bend left, line]   (7) edge (6);
\path [bend left, line]   (9) edge (8);
\path [bend left, line]   (11) edge (10);
\end{scope}

\end{scope}

\begin{scope}[xshift = 5cm,scale = .35]
\begin{knot}[ 	
	%draft mode = crossings,
	consider self intersections,
 	clip width = 6,
 	ignore endpoint intersections = true,
	end tolerance = 2pt
	]
	\flipcrossings{2,4,6,8,9,11,13,15}
	\strand[thick, rounded corners = 1mm]  (3,0) -- (0,0) -- (0,1) -- (1,2) -- (2,1) -- (3,2) -- (1,4) -- (3,6) -- (2,7) -- (1,6) --(0,7) -- (0,8) -- (9.5,8) -- (9.5,5.5) -- (8.5,4.5) -- (9.5,3.5) -- (8.5,2.5) -- (5.5,5.5) -- (4.5,4.5) --  (5.5,3.5) -- (3,1) -- (2,2) -- (1,1) -- (0,2) -- (2,4) -- (0,6) -- (1,7) -- (2,6) -- (3,7) -- (5.5,4.5) -- (4.5,3.5) -- (5.5,2.5) -- (8.5,5.5) -- (9.5,4.5) -- (8.5,3.5) -- (9.5,2.5) -- (9.5,0) -- (3,0);
	
\end{knot}

\draw[thick, rounded corners = 1mm, green!70!black] (-.5,3.5) -- (-.5,.5) -- (3.5,.5) -- (3.5,7.5) -- (-.5,7.5) -- (-.5,3.5);
\draw[thick, rounded corners = 1mm, green!70!black] (7,6) -- (4,6) -- (4,2) -- (10,2) -- (10,6) -- (7,6);

\end{scope}

\end{tikzpicture}\]
\caption{Left: arcs connecting marked points on non-alternating edges. Right: the alternating tangle decomposition of the diagram from Figure \ref{fig:BM}.}
\label{fig:altdecomp}
\end{figure}
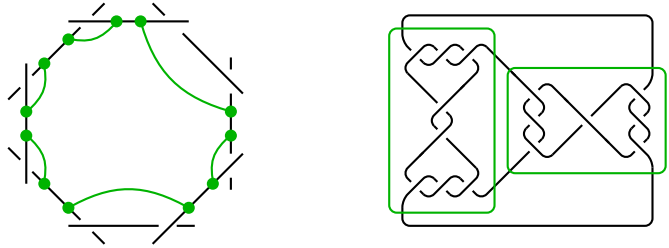

Kim \cite{Kim_2018} and independently Armond and Lowrance \cite{AL_2017} gave the following classification of the links with Turaev genus one.
\begin{theorem}[Kim, Armond, Lowrance]
\label{thm:tg1}
If $L$ is a non-split link with $g_T(L)=1$, then $L$ has a diagram whose alternating tangle decomposition consists of an even number of alternating 4-end tangles connected together in a cycle where every edge is doubled, as in Figure \ref{fig:gt1}.
\end{theorem}
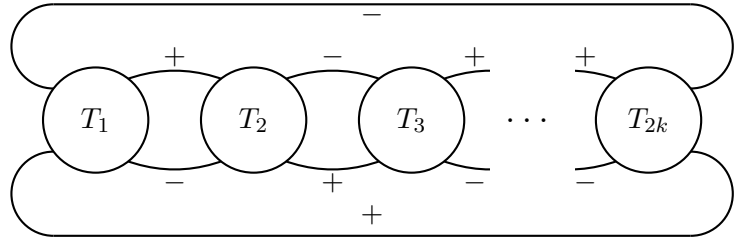
\begin{figure}[h]
\centering
    \[\begin{tikzpicture}[thick,scale = .7]
\draw [bend left] (0,.5) edge (3,.5);
\draw [bend right] (0,-.5) edge (3, -.5);
\draw [bend left] (3,.5) edge (6,.5);
\draw [bend right] (3,-.5) edge (6, -.5);
\draw [bend left] (6,.5) edge (9,.5);
\draw [bend right] (6,-.5) edge (9, -.5);
\draw [bend left] (7.5,.5) edge (10.5,.5);
\draw [bend right] (7.5,-.5) edge (10.5,-.5);

\fill[white] (7.5,1) rectangle (9.1,-1);
\draw (8.25,0) node{\Large{$\dots$}};

\draw (-.8,.6) arc (270:90:.8cm);
\draw (-.8,-.6) arc (90:270:.8cm);
\draw (11.3,.6) arc (-90:90:.8cm);
\draw (11.3,-.6) arc (90:-90:.8cm);
\draw (-.8,2.2) -- (11.3,2.2);
\draw (-.8,-2.2) -- (11.3,-2.2);

\fill[white] (0,0) circle (1cm);
\draw (0,0) node {$T_1$};
\draw (0,0) circle (1cm);
\fill[white] (3,0) circle (1cm);
\draw (3,0) node {$T_2$};
\draw (3,0) circle (1cm);
\fill[white] (6,0) circle (1cm);
\draw (6,0) node {$T_3$};
\draw (6,0) circle (1cm);
\fill[white] (10.5,0) circle (1cm);
\draw (10.5,0) node {$T_{2k}$};
\draw (10.5,0) circle (1cm);

\draw (1.5,1.2) node{$+$};
\draw (1.5,-1.2) node{$-$};
\draw (4.5,1.2) node{$-$};
\draw (4.5,-1.2) node{$+$};
\draw (5.25,2) node{$-$};
\draw (5.25,-1.8)node{$+$};
\draw (7.2,1.2) node{$+$};
\draw (7.2,-1.2) node{$-$};
\draw (9.3,1.2) node{$+$};
\draw (9.3,-1.2)node{$-$};

\end{tikzpicture}\]
    \caption{Every non-split Turaev genus one link has a diagram in the above format. Each tangle $T_i$ is alternating. An edge labeled $+$ corresponds to a non-alternating edge of $D$ whose endpoints are both over-crossings. Similarly, an edge labeled $-$ corresponds to a non-alternating edge of $D$ whose endpoints are both under-crossings.}
    \label{fig:gt1}
\end{figure}
The classification in Theorem \ref{thm:tg1} leads to results about coefficients of the Jones polynomial of a Turaev genus one link \cite{DL_2018, LS_2017}, results about the extremal and near extremal Khovanov homology of Turaev genus one links \cite{DL_2020,BDLMV_2024}, and a proof that no knot or link with Turaev genus one has trivial Jones polynomial \cite{LS_2017}.

We are now ready to give our second proof of Theorem \ref{thm:main}. Since the forward implication was shown in the first proof, we omit it here.
\begin{proof}[Proof of Theorem \ref{thm:main}]
Suppose that $\Span V_L(t) = c(L) - 1$. Since $\Span V_L(t) < c(L)$, it follows that $L$ is non-alternating. Let $D$ be a minimum crossing diagram of $L$. Since $L$ is non-alternating, Remark \ref{rem:extreme} implies that $g_T(D)=1$ and the extreme coefficients $a_m$ and $a_M$ of $\langle D \rangle$ are non-zero.

Consider the following two mutually exclusive cases.
\begin{enumerate}
    \item The decorated all-$A$ state $\widetilde{\sigma_A}(D)$ has no $A$-chords, and the decorated all-$B$ state $\widetilde{\sigma_B}(D)$ has no $B$-chords.
    \item At least one of $\widetilde{\sigma_A}(D)$ and $\widetilde{\sigma_B}(D)$ has at least one $A$- or $B$-chord.
\end{enumerate}
In the first case, the diagram $D$ is adequate, and the proof is complete. In the second case, the diagram $D$ is not adequate, and so we seek a contradiction.

Without loss of generality, suppose that the decorated all-$A$ state $\widetilde{\sigma_A}(D)$ has at least one $A$-chord. Since $a_M$ is nonzero, Lemma \ref{lem:extremecoeff} implies that $\widetilde{\sigma_A}(D)$ has a pair of $A$-chords whose endpoints alternate as one travels around that component of $\sigma_A(D)$. We call this pair of $A$-chords \textit{interleaved}.

Theorem \ref{thm:tg1} implies that $D$ has the form as in Figure \ref{fig:gt1}. The components of the all-$A$ state $\sigma_A(D)$ are of two types. A component is either completely contained inside an alternating tangle $T_i$, or it intersects more than one of the $T_i$'s. If a component $\gamma$ of $\sigma_A(D)$ is completely contained inside $T_i$ for some $i$, then it traces out a face of the tangle diagram because $T_i$ is alternating. Since the $T_i$'s may be taken to have no nugatory crossings, that component $\gamma$ has no $A$-chords, that is, all the $A$-traces meeting $\gamma$ go between two distinct components of $\sigma_A(D)$. 

Figure \ref{fig:Acomps} depicts the components of $\sigma_A(D)$ that intersect more than one $T_i$. If there are more than two alternating tangles in the alternating decomposition of $D$, then it is impossible for any component of $\sigma_A(D)$ to have interleaved $A$-chords. Therefore, the number of alternating tangles in $D$ is two.

Since the diagram $D$ has two alternating tangles and a pair of interleaving $A$-chords, it has the form of the first diagram in Figure \ref{fig:2tang}, where each $R_i$ is an alternating tangle. Figure \ref{fig:2tang} shows that after two flypes and an isotopy, the diagram $D$ is transformed into an alternating diagram $D'$ of $L$. This contradicts the fact that $L$ is non-alternating. Hence case (2) from above cannot occur, and thus $L$ is adequate.
\end{proof}
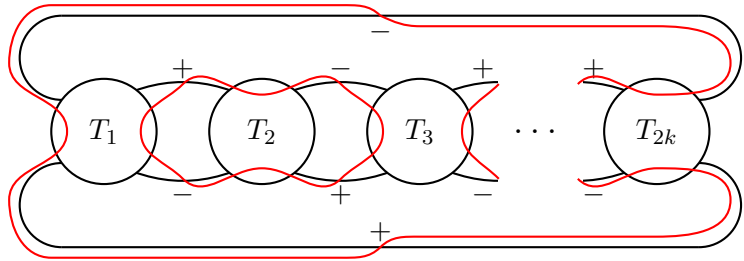
\begin{figure}[h]
\centering
    \[\begin{tikzpicture}[thick,scale = .7]
\draw [bend left] (0,.5) edge (3,.5);
\draw [bend right] (0,-.5) edge (3, -.5);
\draw [bend left] (3,.5) edge (6,.5);
\draw [bend right] (3,-.5) edge (6, -.5);
\draw [bend left] (6,.5) edge (9,.5);
\draw [bend right] (6,-.5) edge (9, -.5);
\draw [bend left] (7.5,.5) edge (10.5,.5);
\draw [bend right] (7.5,-.5) edge (10.5,-.5);

\fill[white] (7.5,1) rectangle (9.1,-1);
\draw (8.25,0) node{\Large{$\dots$}};

\draw (-.8,.6) arc (270:90:.8cm);
\draw (-.8,-.6) arc (90:270:.8cm);
\draw (11.3,.6) arc (-90:90:.8cm);
\draw (11.3,-.6) arc (90:-90:.8cm);
\draw (-.8,2.2) -- (11.3,2.2);
\draw (-.8,-2.2) -- (11.3,-2.2);

\fill[white] (0,0) circle (1cm);
\draw (0,0) node {$T_1$};
\draw (0,0) circle (1cm);
\fill[white] (3,0) circle (1cm);
\draw (3,0) node {$T_2$};
\draw (3,0) circle (1cm);
\fill[white] (6,0) circle (1cm);
\draw (6,0) node {$T_3$};
\draw (6,0) circle (1cm);
\fill[white] (10.5,0) circle (1cm);
\draw (10.5,0) node {$T_{2k}$};
\draw (10.5,0) circle (1cm);

\draw (1.5,1.2) node{$+$};
\draw (1.5,-1.2) node{$-$};
\draw (4.5,1.2) node{$-$};
\draw (4.5,-1.2) node{$+$};
\draw (5.25,1.9) node{$-$};
\draw (5.25,-1.9)node{$+$};
\draw (7.2,1.2) node{$+$};
\draw (7.2,-1.2) node{$-$};
\draw (9.3,1.2) node{$+$};
\draw (9.3,-1.2)node{$-$};

\draw[thick, red] (3,.7) to [out = 180, in  = 45]
(1.5,.9) to [out = 225, in = 90]
(.7,0) to [out = 270, in = 135]
(1.5,-.9) to [out = 315, in = 180]
(3,-.7) to [out = 0, in = 225]
(4.5,-.9) to [out = 45, in = 270]
(5.3,0) to [out = 90, in =315 ]
(4.5,.9) to [out = 135, in= 0]
(3,.7);

\draw[thick,red] (7.5,.9) to [out = 225, in = 90]
(6.8,0) to [out = 270, in =135]
(7.5,-.9);
\draw[thick,red] (9,.9) to [out =45, in = 180]
(10.5,.7) to [out = 0, in = 270]
(11.9, 1.3) to [out = 90, in = 0]
(10.5,2) to [out = 180, in =0]
(6,2) to [out = 180, in = 335]
(5.25, 2.2) to [out = 135, in = 0]
(4.5,2.4) to [out = 180, in = 0]
(-1,2.4) to [out = 180, in = 90]
(-1.8,1.3) to [out =270, in = 90]
(-.7,0) to [out = 270, in = 90]
(-1.8,-1.3) to [out = 270, in = 180]
(-1,-2.4) to [out =0, in = 180]
(4.5,-2.4) to [out = 0, in = 225]
(5.25,-2.2) to [out = 45, in = 180]
(6,-2) to [out = 0, in = 180]
(10.5,-2) to [out = 0, in = 270]
(11.9,-1.3) to [out = 90, in = 0]
(10.5,-.7) to [out = 180, in = 315]
(9,-.9);

\end{tikzpicture}\]
    \caption{Components of the all-$A$ state that are not completely contained in an alternating tangle $T_i$ are drawn in red.}
    \label{fig:Acomps}
\end{figure}

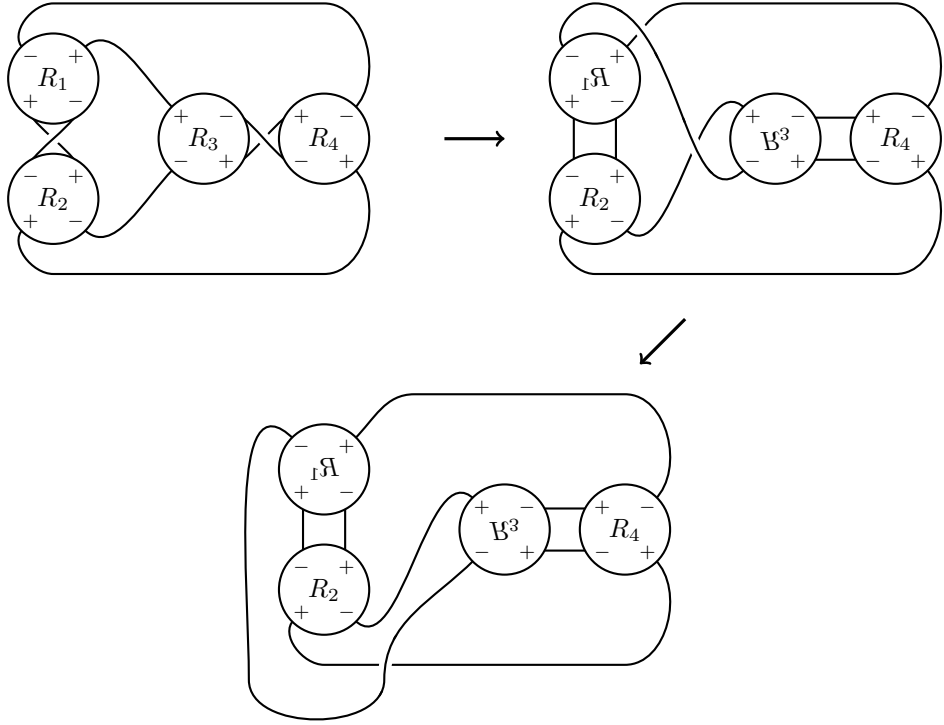
\begin{figure}[h]
\centering
    \[\begin{tikzpicture}[scale=.4, thick]

\begin{knot}[	
	%draft mode = crossings,
	consider self intersections,
 	clip width = 6,
 	ignore endpoint intersections = true,
	end tolerance = 2pt
	]
	\flipcrossings{1}
	\strand[thick]  (-1.2,2) to [out = 270, in = 90] (1.2,-2);
	\strand[thick] (-1.2,-2) to [out = 90, in = 270] (1.2,2);
	\strand[thick] (5,1.2) to [out = 0, in = 180] (9,-1.2);
	\strand[thick] (5,-1.2) to [out = 0, in = 180] (9,1.2);
\end{knot}
\draw (0,2) to [out = 45, in = 225]
(1,3) to [out = 45, in = 135]
(4,1) to [out = 315, in = 135]
(5,0);
\draw (0,-2) to [out = 315, in = 135]
(1,-3) to [out = 315, in = 225]
(4,-1) to [out = 45, in = 225]
(5,0);
\draw (0,2) to [out = 135, in = 315] 
(-1,3) to [out = 135, in = 180]
(0,4.5) to [out = 0, in = 180]
(9,4.5) to [out = 0, in = 45]
(10,1) to [out = 225, in = 45]
(9,0);
\draw (0,-2) to [out = 225, in = 45]
(-1,-3) to [out = 225, in = 180]
(0,-4.5) to [out = 0, in = 180]
(9,-4.5) to [out = 0, in = 315]
(10,-1) to [out = 135, in = 315]
(9,0);

\fill[white] (0,2) circle (1.5cm);
\fill[white] (0,-2) circle (1.5cm);

\fill[white] (5,0) circle (1.5cm);
\fill[white](9,0) circle (1.5cm);

\draw (0,2) circle (1.5cm);
\draw (0,-2) circle (1.5cm);

\draw (5,0) circle (1.5cm);
\draw (9,0) circle (1.5cm);

\draw (0,2) node{\small{$R_1$}};
\draw (0,-2) node{\small{$R_2$}};
\draw (5,0) node{\small{$R_3$}};
\draw (9,0) node{\small{$R_4$}};
\draw (.75,2.75)  node{\tiny{$+$}};
\draw (-.75,2.75) node{\tiny{$-$}};
\draw (.75, 1.25) node{\tiny{$-$}};
\draw (-.75, 1.25) node{\tiny{$+$}};

\draw (.75,-2.75)  node{\tiny{$-$}};
\draw (-.75,-2.75) node{\tiny{$+$}};
\draw (.75, -1.25) node{\tiny{$+$}};
\draw (-.75, -1.25) node{\tiny{$-$}};

\draw (4.25,.75) node{\tiny{$+$}};
\draw (5.75,.75) node{\tiny{$-$}};
\draw (5.75,-.75) node{\tiny{$+$}};
\draw (4.25,-.75) node{\tiny{$-$}};

\draw (8.25,.75) node{\tiny{$+$}};
\draw (9.75,.75) node{\tiny{$-$}};
\draw (9.75,-.75) node{\tiny{$+$}};
\draw (8.25,-.75) node{\tiny{$-$}};

%%%%%%%%%%%%%%%%%%%%

\begin{scope}[xshift = 18 cm]
\begin{knot}[	
	%draft mode = crossings,
	consider self intersections,
 	clip width = 6,
 	ignore endpoint intersections = true,
	end tolerance = 2pt
	]
	%\flipcrossings{}
	\strand[thick]  (0,2) to [out = 135, in = 315]
(-1,3) to [out = 135, in = 180]
(0,4.5) to [out = 0, in = 225]
(5,-1) to [out = 45, in = 225]
(6,0);
\strand[thick] (0,2) to [out = 45, in = 225]
(1,3) to [out = 45, in = 180]
(3,4.5) to [out = 0, in = 180]
(10,4.5) to [out = 0, in = 45]
(11,1) to [out = 225, in = 45]
(10,0);

\strand[thick] (0,-2) to [out = 315, in = 135]
(1,-3) to [out = 315, in = 135]
(5,1) to [out = 315, in = 135]
(6,0);
	
\end{knot}

\draw[thick] (-.7,2) -- (-.7,-2);
\draw[thick] (.7,2) -- (.7,-2);
\draw[thick] (6,.7) -- (10,.7);
\draw[thick] (6,-.7) -- (10,-.7);

\draw (0,-2) to [out = 225, in = 45]
(-1,-3) to [out = 225, in = 180]
(0,-4.5) to [out = 0, in = 180]
(10,-4.5) to [out = 0, in = 315]
(11,-1) to [out = 135, in = 315]
(10,0);

\fill[white] (0,2) circle (1.5cm);
\fill[white] (0,-2) circle (1.5cm);

\fill[white] (6,0) circle (1.5cm);
\fill[white](10,0) circle (1.5cm);

\draw (0,2) circle (1.5cm);
\draw (0,-2) circle (1.5cm);

\draw (6,0) circle (1.5cm);
\draw (10,0) circle (1.5cm);

\draw (0,2) node{\small{\reflectbox{$R_1$}}};
\draw (0,-2) node{\small{$R_2$}};
\draw (6,0) node{\small{\raisebox{\depth}{\scalebox{1}[-1]{$R_3$}}}};
\draw (10,0) node{\small{$R_4$}};
\draw (.75,2.75)  node{\tiny{$+$}};
\draw (-.75,2.75) node{\tiny{$-$}};
\draw (.75, 1.25) node{\tiny{$-$}};
\draw (-.75, 1.25) node{\tiny{$+$}};

\draw (.75,-2.75)  node{\tiny{$-$}};
\draw (-.75,-2.75) node{\tiny{$+$}};
\draw (.75, -1.25) node{\tiny{$+$}};
\draw (-.75, -1.25) node{\tiny{$-$}};

\draw (5.25,.75) node{\tiny{$+$}};
\draw (6.75,.75) node{\tiny{$-$}};
\draw (6.75,-.75) node{\tiny{$+$}};
\draw (5.25,-.75) node{\tiny{$-$}};

\draw (9.25,.75) node{\tiny{$+$}};
\draw (10.75,.75) node{\tiny{$-$}};
\draw (10.75,-.75) node{\tiny{$+$}};
\draw (9.25,-.75) node{\tiny{$-$}};

\end{scope}

%%%%%%%%%%%%%%%%%%%%

\begin{scope}[xshift = 9 cm, yshift = -13 cm]
\begin{knot}[	
	%draft mode = crossings,
	consider self intersections,
 	clip width = 6,
 	ignore endpoint intersections = true,
	end tolerance = 2pt
	]
	%\flipcrossings{}
	\strand[thick]  (0,2) to [out = 135, in = 315]
(-1,3) to [out = 135, in = 90]
(-2.5, -5) to [out =270, in = 270]
(2,-5) to [out = 90, in = 225]
(5,-1) to [out = 45, in = 225]
(6,0);
\strand[thick] (0,2) to [out = 45, in = 225]
(1,3) to [out = 45, in = 180]
(3,4.5) to [out = 0, in = 180]
(10,4.5) to [out = 0, in = 45]
(11,1) to [out = 225, in = 45]
(10,0);

\strand[thick] (0,-2) to [out = 315, in = 135]
(1,-3) to [out = 315, in = 135]
(5,1) to [out = 315, in = 135]
(6,0);

\strand[thick] (0,-2) to [out = 225, in = 45]
(-1,-3) to [out = 225, in = 180]
(0,-4.5) to [out = 0, in = 180]
(10,-4.5) to [out = 0, in = 315]
(11,-1) to [out = 135, in = 315]
(10,0);
	
\end{knot}

\draw[thick] (-.7,2) -- (-.7,-2);
\draw[thick] (.7,2) -- (.7,-2);
\draw[thick] (6,.7) -- (10,.7);
\draw[thick] (6,-.7) -- (10,-.7);

\fill[white] (0,2) circle (1.5cm);
\fill[white] (0,-2) circle (1.5cm);

\fill[white] (6,0) circle (1.5cm);
\fill[white](10,0) circle (1.5cm);

\draw (0,2) circle (1.5cm);
\draw (0,-2) circle (1.5cm);

\draw (6,0) circle (1.5cm);
\draw (10,0) circle (1.5cm);

\draw (0,2) node{\small{\reflectbox{$R_1$}}};
\draw (0,-2) node{\small{$R_2$}};
\draw (6,0) node{\small{\raisebox{\depth}{\scalebox{1}[-1]{$R_3$}}}};
\draw (10,0) node{\small{$R_4$}};
\draw (.75,2.75)  node{\tiny{$+$}};
\draw (-.75,2.75) node{\tiny{$-$}};
\draw (.75, 1.25) node{\tiny{$-$}};
\draw (-.75, 1.25) node{\tiny{$+$}};

\draw (.75,-2.75)  node{\tiny{$-$}};
\draw (-.75,-2.75) node{\tiny{$+$}};
\draw (.75, -1.25) node{\tiny{$+$}};
\draw (-.75, -1.25) node{\tiny{$-$}};

\draw (5.25,.75) node{\tiny{$+$}};
\draw (6.75,.75) node{\tiny{$-$}};
\draw (6.75,-.75) node{\tiny{$+$}};
\draw (5.25,-.75) node{\tiny{$-$}};

\draw (9.25,.75) node{\tiny{$+$}};
\draw (10.75,.75) node{\tiny{$-$}};
\draw (10.75,-.75) node{\tiny{$+$}};
\draw (9.25,-.75) node{\tiny{$-$}};

\end{scope}

\draw[very thick, ->] (13,0) -- (15,0);
\draw[very thick, ->] ( 21,-6) -- (19.5,-7.5);

\end{tikzpicture}\]
    \caption{In each diagram, the tangles $R_i$ are alternating. A +/- label indicates the first crossing of the strand is an over/under-crossing respectively. The first diagram is a generic diagram $D$ where $g_T(D)=1$ and $\widetilde{\sigma_A}(D)$ has a pair of interleaving $A$-chords. Applying two flypes to the first diagram yields the second, and an additional isotopy yields the third.}
    \label{fig:2tang}
\end{figure}

\section{Consequences and applications}
\label{sec:cor}
In this section, we discuss some consequences and applications of Theorem \ref{thm:main}. The first proof of Theorem \ref{thm:main} gives expressions for $\min \deg_a \Lambda_D(a,z)$ and $\max \deg_a \Lambda_D(a,z)$.

\begin{corollary}
\label{cor:adeg}
If $D$ is an adequate link diagram with $g_T(D)=1$, then $\max\deg_{a}\Lambda_{D}(a,z) =  |\sigma_{B}(D)| - 1$, and $\min\deg_{a}\Lambda_{D}(a,z) =  -|\sigma_{A}(D)| + 1$.  Therefore,  we obtain $\Span_{a}(\Lambda_{D}(a,z)) = \Span_{a}(F_{L}(a,z)) =  c(L) - 2 = |\sigma_{A}(D)| + |\sigma_{B}(D)| -  2$.
\end{corollary}
Yokota \cite{Yokota_1995} proved that $\max\deg_{a}\Lambda_D(a,z) = |\sigma_B(D)|-1$ and $\min\deg_a \Lambda_D(a,z)=-|\sigma_A(D)|+1$ for alternating links. Corollary \ref{cor:adeg} can be seen as an extension of Yokota's result to adequate Turaev genus one links.

In the next corollary, we conclude that links satisfying $\Span V_L(t) = c(L) - 1$ have various properties. Before stating the corollary, we define or give references describing these properties. Ozsv\'ath and Szab\'o \cite{OS_2005} defined \textit{quasi-alternating links} as the smallest set $Q$ of links containing the unknot and satisfying the condition that if $\tikz[baseline=.6ex, scale = .4]{
\draw[rounded corners = 1.5mm] (0,0) -- (.45,.5) -- (0,1);
\draw[rounded corners = 1.5mm] (1,0) -- (.55,.5) -- (1,1);
}$ and $ \tikz[baseline=.6ex, scale = .4]{
\draw[rounded corners = 1.5mm] (0,0) -- (.5,.45) -- (1,0);
\draw[rounded corners = 1.5mm] (0,1) -- (.5,.55) -- (1,1);
}$ are in $Q$ and $\det \tikz[baseline=.6ex, scale = .4]{
\draw[rounded corners = 1.5mm] (0,0) -- (.45,.5) -- (0,1);
\draw[rounded corners = 1.5mm] (1,0) -- (.55,.5) -- (1,1);
} + \det \tikz[baseline=.6ex, scale = .4]{
\draw[rounded corners = 1.5mm] (0,0) -- (.5,.45) -- (1,0);
\draw[rounded corners = 1.5mm] (0,1) -- (.5,.55) -- (1,1);
} = \det \tikz[baseline=.6ex, scale = .4]{
\draw (0,0) -- (1,1);
\draw (1,0) -- (.7,.3);
\draw (.3,.7) -- (0,1);
}$, then $\tikz[baseline=.6ex, scale = .4]{
\draw (0,0) -- (1,1);
\draw (1,0) -- (.7,.3);
\draw (.3,.7) -- (0,1);
}$ is also in $Q$. The arc index of a link was defined in Section \ref{sec:proof1}. The Jones diameter $jd_L$ of a link $L$ is a measure of the span of the colored Jones polynomial of $L$; see \cite{KL_2023} for a precise definition. A $4$-end tangle is \textit{strongly alternating} if both its natural closures are reduced alternating diagrams. Lickorish and Thistlethwaite \cite{LT_1988} define a \textit{semi-alternating link} to be a link with a diagram obtained by summing two strongly alternating tangles to obtain a non-alternating tangle and taking the closure as in Figure \ref{fig:semi}. A semi-alternating link is adequate and has Turaev genus one.

\begin{figure}[h]
\[\begin{tikzpicture}[thick]

\begin{scope}[scale = .8]
\draw (2^-.5,2^-.5) to [out = 45, in = 0] (0,1.5) to [out = 180, in = 135] (-2^-.5,2^-.5);
\draw (2^-.5,-2^-.5) to [out = 315, in = 0] (0, -1.5) to [out =180, in = 225] (-2^-.5,-2^-.5);

\fill[white] (0,0) circle (1cm);
\draw (0,0) circle (1cm);
\draw (0,0) node{$T$};

\begin{scope}[xshift = 5cm]

\draw (2^-.5,2^-.5) to [out = 45, in = 90] (1.5,0) to [out = 270, in = 315] (2^-.5,-2^-.5);
\draw (-2^-.5,2^-.5) to [out = 135, in = 90] (-1.5,0) to [out =270, in = 225] (-2^-.5,-2^-.5);

\fill[white] (0,0) circle (1cm);
\draw (0,0) circle (1cm);
\draw (0,0) node{$T$};

\end{scope}
\end{scope}

\begin{scope}[scale = .6, yshift = -8.5cm, xshift = -.4cm]

\begin{knot}[	
	%draft mode = crossings,
	consider self intersections,
 	clip width = 3,
 	ignore endpoint intersections = true,
	end tolerance = 2pt
	]
	\flipcrossings{2,9,5,3,7}
	\strand[rounded corners = 2mm]
	(4,0) -- (0,0) -- (0,1) -- (1,2) -- (1,3) -- (0,4) -- (0,5) -- (8,5) -- (8,4) -- (7,3) -- (6,4) -- (5,3) -- (5,2) -- (6,1) -- (7,1) -- (8,2) -- (8,3) -- (7,4) -- (6,3) -- (5,4) -- (3,4) -- (2,3) -- (3,2) -- (2,1) -- (1,1) -- (0,2) -- (0,3) -- (1,4) -- (2,4) -- (3,3) -- (2,2) -- (3,1) -- (5,1) -- (6,2) -- (7,2) -- (8,1) -- (8,0) -- (4,0);

\end{knot}

\draw[thick, dashed, green!70!black] (1.5,2.5) circle (2cm);
\draw[thick, dashed, green!70!black] (6.5,2.5) circle (2cm);

\end{scope}

\end{tikzpicture}\]
\label{fig:semi}
\caption{Top: The two natural closures of the tangle $T$. If both are reduced alternating diagrams, then $T$ is strongly alternating. Bottom: Two strongly alternating tangles combined together to give a semi-alternating link.}
\end{figure}
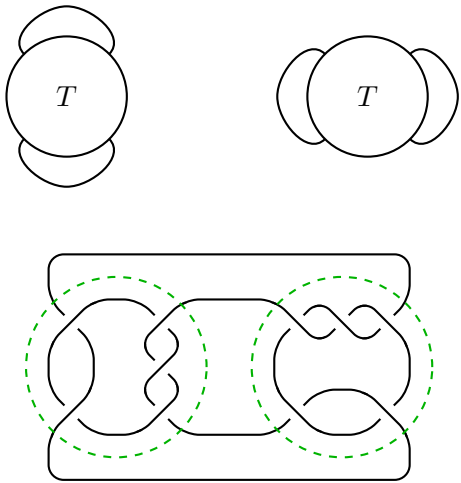

\begin{corollary}
\label{cor:AdeqProperties}
Let $L$ be a non-split link such that $\Span V_{L}(t) = c(L) - 1$. Then $L$ is not quasi-alternating, has arc index $\alpha(L) =c(L)$, has Jones diameter $jd_{L} = 2c(L)$, and the leading and trailing coefficients of $V_{L}(t)$ have absolute value one. In particular, any semi-alternating link satisfies all of above properties.
\end{corollary}
\begin{proof}
Since $\Span V_L(t) < c(L)$, it follows that $L$ is non-alternating. By Theorem \ref{thm:main}, $L$ is adequate and has Turaev genus one. Khovanov \cite{Kh_2003} proved that non-alternating adequate links have thick Khovanov homology. Since quasi-alternating links have thin Khovanov homology \cite{MO_2007}, it follows that $L$ is not quasi-alternating. Corollary \ref{cor:adeg} implies that $\Span_a \Lambda_D(a,z) = c(D)-2$. Morton and Beltrami \cite{MB_1998} proved that $\Span_a \Lambda_D(a,z) + 2 \leq \alpha(L)$, and Jin and Park \cite{JP_2010} proved that $\alpha(L)\leq c(L)$. Therefore $\alpha(L) = c(L)$. Kalfagianni and Lee \cite{KL_2023} proved that $jd_{L} = 2c(L)$ when $L$ is adequate. Since $L$ is adequate, the leading and trailing coefficients of $V_L(t)$ have absolute value one by Lickorish and Thistlethwaite \cite{LT_1988}.
\end{proof}
Lickorish and Thistlethwaite \cite{LT_1988} proved that the coefficients of the Jones polynomial of a semi-alternating link do not alternate in sign. Since the Khovanov homology of a quasi-alternating link is thin, the coefficients of its Jones polynomial alternate in sign. This gives another proof that semi-alternating links are not quasi-alternating. In forthcoming work, del Valle V\'ilchez and Lowrance \cite{DL_2025} show that the arc index of any adequate link is given by $\alpha(L) = c(L) - 2g_T(L) + 2 = |\sigma_A(D)| + |\sigma_B(D)|$, where $D$ is an adequate diagram of $L$.

The next corollary is similar in spirit to \cite[Corollary\,2.1]{T_1988_3}. 
\begin{corollary}
Let $D$ be a prime link diagram such that $\Span \langle D \rangle= 4(c(D)-1)$. Then either $D$ is an adequate diagram with $g_T(D)=1$ or $D$ represents an alternating link with crossing number $c(D)-1$. Therefore, the leading and trailing coefficients of $\langle D \rangle$ are $\pm 1$.
\end{corollary}
\begin{proof}
The corollary follows immediately from Theorem \ref{thm:main} and the fact that $4(c(D)-1) = \Span \langle D \rangle \leq 4c(L)$. 
\end{proof}

\begin{corollary}
Let $L$ be a quasi-alternating link. Either $c(L) = \Span V_{L}(t)$ or $ \Span V_{L}(t) \leq c(L)-2$.
\end{corollary}

\begin{proof}
The corollary follows immediately from Corollary \ref{cor:AdeqProperties} and the fact that  $c(L) = \Span V_{L}(t)$ if and only if $L$ is alternating.
\end{proof}

Many knots in the tables satisfy $\Span V_K(t) = c(K)-1$, and hence are adequate. No knots with crossing number at most nine satisfy the equation. Table \ref{tab:number} shows the number of prime knots with crossing number between $10$ and $16$ that satisfy $\Span V_K(t) = c(K)-1$. See \cite{knotinfo} for a complete list of adequate knots with at most $13$ crossings. 

\begin{table}

\begin{tabular}{| c | c |}
\hline
Crossing number & Number of knots\\
\hline
\hline
10 & 3\\
\hline
11 & 15\\
\hline
12 & 78\\
\hline
13 & 517\\
\hline
14 & 2668\\
\hline
15 & 15365\\
\hline
16 & 82505\\
\hline
\end{tabular}
\caption{The number of prime knots with a given crossing number satisfying $\Span V_K(t) = c(K) -1$.}
\label{tab:number}
\end{table}

\bibliographystyle{amsalpha}
\bibliography{QC}

\end{document}